\documentclass[oneside]{amsart}
\usepackage[utf8]{inputenc}
\usepackage{microtype}
\usepackage{hyperref}
\usepackage{tikz,tikz-cd, xypic}
\usepackage[]{todonotes}
\usepackage{geometry}

	\usepackage[frenchmath]{mathastext}
	\usepackage{MnSymbol}
	\usepackage{eucal}
	\usepackage{amsthm}
	\usepackage{thmtools}
	\numberwithin{equation}{section}
	\theoremstyle{definition}

	\newtheorem*{lem*}{Lemma}
	\newtheorem*{thrm*}{Theorem}
	\newtheorem*{ex*}{Example}
	\newtheorem*{fact*}{Fact}
	\newtheorem*{not*}{Notation}
	\newtheorem*{defn*}{Definition}
	\newtheorem*{prop*}{Proposition}
	\newtheorem*{rmk*}{Remark}
	
	\newtheorem{defn}{Definition}[section]

	\newtheorem{prop}[defn]{Proposition}
	
	\newtheorem{rmk}[defn]{Remark}
	\newtheorem{thrm}[defn]{Theorem}
	\newtheorem{cor}[defn]{Corollary}
	\newtheorem{lem}[defn]{Lemma}

	\newcommand{\mathcats}{\mathcal} %per ora mi è presa la fissazione di usare questo font per le categorie, ma non so ancora se mi piaccia davvero
	\newcommand{\cat}{\mathcal}

	\newcommand{\C}{\mathcats C}
	\DeclareMathOperator{\Ann}{Ann}
	\DeclareMathOperator{\Aut}{Aut}
	\DeclareMathOperator{\Coh}{Coh}
	\DeclareMathOperator{\codim}{codim}
	\DeclareMathOperator{\coker}{coker}
	\DeclareMathOperator{\End}{End}

	\DeclareMathOperator{\Hom}{Hom}
	\DeclareMathOperator{\Id}{Id}

	\renewcommand{\mod}{\text{mod}}
	\makeatletter\let\O\@undefined\makeatother
	\newcommand{\O}{\mathcal O}
	
	\DeclareMathOperator{\Pic}{Pic}
	
	\DeclareMathOperator{\Spec}{Spec}
	\DeclareMathOperator{\supp}{supp}

	\newcommand{\N}{\mathbf N}
	
	\makeatletter\let\k\@undefined\makeatother
	\newcommand{\k}{\mathbf{k}}

	\newcommand{\m}{\mathfrak{m}}
	
	\makeatletter\let\P\@undefined\makeatother
	\newcommand{\P}{\mathbf{P}}

	\newcommand{\qto}{\rightharpoonup}

	\DeclareMathOperator{\lHom}{\underline{Hom}}

	\newcommand{\into}{\hookrightarrow}
	\newcommand{\onto}{\twoheadrightarrow}
	
\begin{document}
\begin{abstract}
	Extending work of Meinhardt and Partsch, we prove that two varieties are isomorphic in codimension $c$ if and only if certain quotients of their categories of coherent sheaves are equivalent.
	This result interpolates between Gabriel's reconstruction theorem and the fact that two varieties are birational if and only if they have the same function field.
\end{abstract}
\title{Gabriel's theorem and birational geometry}
\author{John Calabrese, Roberto Pirisi}
\date{\today}
\maketitle

%!TEX root = MASTER.tex
It is a well known fact that two varieties (i.e. irreducible and reduced schemes of finite type over a field $\k$) $X$ and $Y$ are birational if and only if their function fields are isomorphic.
At the same time, a theorem of Gabriel says that $X$ and $Y$ are isomorphic if and only if their categories of coherent sheaves are equivalent (as $\k$-linear categories).
In this article we show that these results are actually related: they are the two extreme cases of our main theorem.

Before giving a precise statement, we introduce some notation.
For an integer $k$, we write $\Coh_{\leq k}(X) \subset \Coh(X)$ for the subcategory of sheaves supported in dimension at most $k$.
There is a robust theory of quotients of abelian categories, and we define $\C_k(X) \coloneq \Coh(X)/\Coh_{\leq k-1}(X)$.
It is often convenient to re-index these categories by codimension, defining $\C^c(X) \coloneq \C_{\dim X - c}(X)$.
We have $\C^{\dim X}(X) = \Coh(X)$, and one shows that $\C^0(X)$ is equivalent to finite-dimensional vector spaces over the function field of $X$.
Finally, recall that two varieties $X$ and $Y$ are \emph{isomorphic in codimension $c$} if there exist open subsets $U \subset X$, $V \subset Y$ such that $U$ is isomorphic to $V$, and the codimensions of $X \setminus U$ and $Y \setminus V$ are at least $c$.
In particular, $X$ and $Y$ are birational if and only if they are isomorphic in codimension zero.
\begin{thrm*}
	Two varieties $X$ and $Y$ are isomorphic in codimension $c$ if and only if the categories $\C^c(X)$, $\C^c(Y)$ are equivalent.
\end{thrm*}
This result actually holds true in broad generality.
For example we do not need to assume our schemes to be reduced or irreducible.
See Theorem \ref{suca} for the general result we prove.

As in Gabriel's theorem for $\Coh(X)$, we can also characterize the group of autoequivalences.
Let $\Aut_{\geq k}(X)$ denote the group of birational self-maps, which are defined away from a closed subset of dimension at most $k-1$.
Let $\Pic_{\geq k}(X)$ denote the group (under tensor products) of sheaves which are invertible away from a closed subset of dimension at most $k-1$ (see \ref{picardo}).
The group $\Aut_{\geq k}(X)$ acts on $\Pic_{\geq k}(X)$ by $(f,L) \mapsto (f^{-1})^* L$.
We may thus form the semi-direct product $\Aut_{\geq k}(X) \ltimes \Pic_{\geq k}(X)$.
\begin{thrm*}
	Let $X$ be a variety.
	There is an isomorphism between $\Aut(\C_k(X))$, the group of autoequivalences, and $\Aut_{\geq k}(X) \ltimes \Pic_{\geq k}(X)$.
\end{thrm*}
Once again, the theorem holds in broad generality.
See Theorem \ref{autoteorema}.

\subsection*{Previous work}
Gabriel originally proved his reconstruction theorem in \cite{gabriel}.
This was later generalized considerably by various people \cite{rosenberg, brand, jm}.
In \cite{mp}, Meinhardt and Partsch were interested in constructing stability conditions on the (derived categories of the) quotient categories $\C^c(X)$.
Along the way they showed the $c=0$ and $c=1$ cases of our main result, when $X,Y$ are smooth and projective over an algebraically closed field.

\subsection*{Acknowledgements}
We would like to thank David Rydh for reading a first draft, and providing excellent comments and corrections.
In particular, we thank him for pointing out a mistake in the proof of lemma \ref{nosheaf} and for noticing that the hypotheses of our main theorem could be relaxed (see Remark \ref{Rydh}).
We thank Pieter Belmans, Zinovy Reichstein, and Angelo Vistoli for useful comments.
Finally, JC wishes to give extra-warm thanks to Richard ``TDR'' Shadrach for many many conversations throughout the years.

\subsection*{Conventions}
For the entirety of this paper we fix a base (commutative) noetherian ring $\k$.
All algebras, schemes, categories, morphisms, and functors are implicitly assumed to be over $\k$.
Starting from Section \ref{gabrielsection}, all schemes will be assumed to be of finite type over $\k$.
If $X$ is a scheme, $U \subset X$ is open, and $F$ is a sheaf on $X$, we will write $F_U$ for the restriction (i.e.\ pullback) of $F$ to $U$.
If $R$ is a ring, we will write $\mod(R)$ for the category of finitely generated $R$-modules.

\setcounter{tocdepth}{1}
\tableofcontents
%!TEX root = MASTER.tex
\section{Quotient categories}
We begin by briefly reviewing some standard notions and later introduce some notation.
We refer to \cite{gabriel,faith} for a thorough treatment of quotients of abelian categories.
Let $\cat A$ be an abelian category.
A subcategory $\cat S \subset \cat A$ is \emph{Serre} if, given an exact sequence $A \to B \to C$, then $B \in \cat S$ if and only if $A,C \in \cat S$.
Given a Serre subcategory $\cat S \subset \cat A$, we may form the quotient $\Phi\colon \cat A \to \cat A / \cat S$.
This category is universal among all abelian categories $\cat B$ equipped with an exact functor $\Psi\colon \cat A \to \cat B$ such that $\Psi(S) = 0$ for all $S \in \cat S$.
Recall that the \emph{kernel} of a functor $\Phi$ is the full subcategory whose objects satisfy $\Phi(M) = 0$.
Thus $\cat S$ is precisely the kernel of $\cat A \to \cat A/ \cat S$.

\begin{lem}\label{iso}
	Let $\cat S_1 \subset \cat S_2 \subset \cat A$ be two Serre subcategories of the abelian category $\cat A$.
	Then $\cat S_2 / \cat S_1$ is a Serre subcategory of $\cat A/ \cat S_1$ and the quotient $\cat A / \cat S_2$ is naturally equivalent to the iterated quotient $\left( \cat A / \cat S_1 \right) / \left( \cat S_2 / \cat S_1 \right)$.
\end{lem} 
\begin{proof}
	The first assertion follows from the definitions, while the second is a consequence of the universal property.
\end{proof}
	Let $\cat A / \cat S$ be a quotient category.
	Let $P, Q \in \cat A$.
	To distinguish arrows in $\cat A$ from arrows in $\cat A / \cat S$, we will write $P \to Q$ for the former and $P \qto Q$ for the latter.
	If two objects $P,Q \in \cat A$ become isomorphic in $\cat A/ \cat S$, we will write $P \equiv Q$.
	Finally, a morphism $f\colon P \to Q$ such that $\ker f, \coker f \in \cat S$ will be called a \emph{weak equivalence}.
	
The quotient category $\cat A / \cat S$ may be concretely built as follows \cite[Exercise 8.12]{ks}.
The objects of $\cat A / \cat S$ are the same as the objects of $\cat A$.
A morphism $P \qto Q$ is an equivalence class of ``roofs'', i.e. diagrams
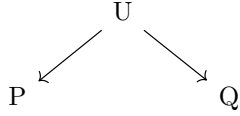
\begin{center}
\begin{tikzcd}
	& U \ar{dl}\ar{dr} & \\
	P && Q 
\end{tikzcd}
\end{center}
with $U \to P$ a weak equivalence.
% Two roofs are equivalent if there is an isomorphism $U \to U'$ making the obvious diagram commute.
% \begin{center}
% \begin{tikzcd}
% 	& U \ar{dl}\ar{dr} \ar{dd} & \\
% 	P && Q \\
% 	& U' \ar{ul}\ar{ur} &
% \end{tikzcd}
% \end{center}
%
% One may also switch from the left to the right convention, and insist that $P \to Q$ be a weak equivalence instead.

\begin{lem}\label{exact}
	Let $\cat A / \cat S$ be a quotient category.
	\begin{itemize}
		\item $P \equiv 0$ if and only if $P \in \cat S$.
		% \item $P \equiv Q$ if and only if there exists an object $R$ and weak equivalences $R \to P$, $R \to Q$.
		\item If $0 \qto A \qto B \qto C \qto 0$ is short exact in $\cat A / \cat S$, then there exists objects $A',B',C'$ and isomorphisms (in $\cat A/ \cat S$) $A \qto A'$, $B \qto B'$, $C \qto C'$ and a short exact sequence $0 \to A' \to B' \to C' \to 0$ in $\cat A$, such that the diagram
		\begin{center}
			\begin{tikzcd}
				0 \ar{r} & A' \ar{r}\ar[rightharpoonup]{d} & B' \ar{r}\ar[rightharpoonup]{d} & C' \ar{r}\ar[rightharpoonup]{d} & 0 \\
				0 \ar{r} & A \ar[rightharpoonup]{r} & B \ar[rightharpoonup]{r} & C \ar{r} & 0 \\
			\end{tikzcd}
		\end{center}
		commutes in $\cat A / \cat S$.
	\end{itemize} 
\end{lem}
\begin{proof}
	The first claim follows from \cite[Prop 7.1.20 (ii)]{ks}, the second is \cite[Cor 15.8]{faith}.
\end{proof}

\subsection{Dimension}\label{subdimension}
Let $\cat A$ be an abelian category.
We say an object $P$ is \emph{minimal} if it has no non-trivial sub-objects.\footnote{What we call minimal objects are often called \emph{simple} objects. However, by a \emph{simple sheaf} one typically means one that has only scalar endomorphisms.}
Note that for us the zero object is also minimal (this is slightly non-standard).
We let $\cat S_0$ be the smallest Serre subcategory containing all minimal objects of $\cat A$.
Let $\cat A_1 = \cat A / \cat S_0$.
Since $\cat A_1$ is also an abelian category, we may repeat the process.
Let $\cat S'_1$ be the smallest Serre subcategory containing all minimal objects of $\cat A_1$ and let $\cat A_2 = \cat A_1 / \cat S'_1$.
We define $\cat S_1$ to be the kernel of $\cat A \to \cat A_1 \to \cat A_2$.
By iterating we obtain a sequence of quotients
\begin{align}\label{quotients_of_cats}
	\cat A = \cat A_0 \onto \cat A_1 \onto \cat A_2 \onto \cat A_3 \onto \cdots \onto \{0\}
\end{align}
and, by taking kernels, we find a nested family of Serre subcategories
\begin{align}\label{kers}
	\{0\} \subset \cat S_0 \subset \cat S_1 \subset \cat S_2 \subset \cat S_3  \subset \cdots \subset \cat A
\end{align}
Of course, it need not be the case that there is a $k$ such that $\cat A_k = \{0\}$ or that $\cat S_k = \cat A$.
\begin{defn}
	The \emph{Krull dimension} of $\cat A$ is
	\begin{align*}
		\dim \cat A \coloneq \inf \{ k \mid \cat A_{k+1} = 0 \} = \inf \{ k \mid \cat S_k = \cat A \}.
	\end{align*}
	We also define
	\begin{align*}
		\dim M \coloneq \inf \{ k \mid M = 0 \text{ in } \cat A_{k+1} \} = \inf \{ k \mid M \in \cat S_k \}
	\end{align*}
	for any object $M \in \cat A$.
\end{defn}
These definitions are of course justified by the algebro-geometric context.% \footnote{We should also mention that there are more refined versions of Krull dimension for abelian categories. }\todo{add citations in footnote}

\subsection{The geometric case}
In this subsection, $X$ denotes an arbitrary noetherian scheme.
Let $\cat C = \Coh(X)$ be its category of coherent sheaves.
As before, we write $\cat S_k = \cat S_k(X)$ for the sequence of kernels as in \eqref{kers}.
Denote by $\Coh_{\leq k}(X)$ the category of sheaves supported in dimension at most $k$.
\begin{prop}\label{geo}
	We have $\cat S_k(X) = \Coh_{\leq k}(X)$.
\end{prop}
Before we prove this result, we need an alternative way to compute morphisms in our quotients.
We define $\cat C_0 = \cat C$ and $\cat C_k = \cat C / \cat S_{k-1}$ for all other $k$.

Let $\Sigma$ be a \emph{family of supports}, i.e.\ a collection of closed subsets of $X$ such that
\begin{itemize}
	\item $\emptyset \in \Sigma$,
	\item if $Z \in \Sigma$, $V \subset Z$ is closed, then $V \in \Sigma$,
	\item if $Z_1, Z_2 \in \Sigma$ then $Z_1 \cup Z_2 \in \Sigma$.
\end{itemize}
For example, $\Sigma$ might be the collection of all closed subsets of $X$ of dimension at most $k$.
Write $\Coh_\Sigma(X) \subset \Coh(X)$ for the subcategory of sheaves $F$ such that $\supp F \in \Sigma$.
Because of our assumptions, $\Coh_\Sigma(X)$ is a Serre subcategory.
We write $\cat Q = \Coh(X)/\Coh_\Sigma(X)$ for the quotient.
\begin{lem}\label{hom}
	Let $F,G \in \Coh(X)$.
	Then
	\begin{align*}
		\Hom_{\cat Q}(F,G) = \varinjlim_{
			\substack{\emptyset \subset U \subset X \\
				 X \setminus U \in \Sigma}}
			\Hom_U(F_U,G_U)
	\end{align*}
	where $F_U, G_U$ denote the restrictions to $U$.
\end{lem}
\begin{proof}
We will define mutually inverse maps, following \cite[Lemma 3.6]{mp}. % $$\Hom_{\cat Q}(F,G) \leftrightarrow \varinjlim_{
 % 			\substack{\emptyset \subset U \subset X \\
 % 				 X \setminus U \in \Sigma}}
 % 			\Hom_U(F|U,G|U)$$ 
We start by constructing a map from right to left.
Any element of the right hand side may be represented as a pair $(U,f)$, where $X \setminus U \in \Sigma$ and $f\colon F_U \to G_U$.
Write $\Gamma_f \subset F_U \oplus G_U$ for the graph of $f$.
Let $E \subset F \oplus G$ be any coherent sheaf such that $E_U = \Gamma_f$.
Then $ E \xrightarrow{\mathrm{pr}_1} F$ is a weak equivalence in $\cat Q$, so the roof $F \xleftarrow{\mathrm{pr}_1} E \xrightarrow{\mathrm{pr}_2} G$ represents an element $f \in \Hom_{\cat Q}(F,G)$.

Suppose now $f'\colon F_{U'} \to G_{U'}$ represents the same element as $(U,f)$.
As above, we construct a roof $F \xleftarrow{\mathrm{pr}'_1} E' \xrightarrow{\mathrm{pr}'_2} G$.
Since $(U,f)$ an $(U',f')$ must eventually agree, there exists an open subset $U''$ with $f_{U''}=f'_{U''}$ and such that $E''=E \cap E'$ is equivalent to both $E$ and $E'$ through the inclusion map.
The commutative diagram
\[ \xymatrix @C=0.5cm @R=0.5cm { & & E'' \ar[dl]_i \ar[dr]^{j} & & \\ & E \ar[dl]_{\mathrm{pr}_1} \ar[drrr]_{\mathrm{pr}_2} & & E' \ar[dlll]^{\mathrm{pr}'_1} \ar[dr]^{\mathrm{pr}'_2} & \\ F & & & & G } \]
shows that the two roofs are equivalent precisely as in \cite[Lemma 3.6]{mp}.

Consider now a morphism $f \in \Hom_{\cat Q}(F,G)$, represented by the roof $F \xleftarrow{s} E \xrightarrow{t} G$.
Since $s$ is a weak equivalence (by definition of roof), we have $\mathrm{supp}(\ker(s)) \cup \mathrm{supp}(\coker(s)) \in \Sigma$.
Let $U$ be its complement.
Once we restrict to $U$, $s$ becomes an isomorphism, hence the map $t \circ s^{-1}\colon F_U \rightarrow G_U$ makes sense.
To make sure the function $f\mapsto t \circ s^{-1}$ is well defined, we use that two different roofs representing the same morphism must be equivalent.
Indeed, up to restricting to a smaller subset, we will obtain the same map.
Finally, since we are allowed to choose ad hoc representatives, one checks the two maps just defined compose to the respective identities.
\end{proof}

%\todo{ci ho pensato, e mi piace tenerlo come lemma a separato}
\begin{lem}\label{min}
	An object $P \in \Coh(X)/\Coh_{\leq k-1}(X)$ is minimal if and only if either $P \equiv 0$ or $P \equiv \O_Z$ for $Z \into X$ an integral closed subscheme of dimension $k$.
\end{lem}
\begin{proof}
	Let us write $\cat Q = \Coh(X)/\Coh_{\leq k-1}(X)$.
	We begin by showing that $\O_Z$ is minimal, for $Z \into X$ an integral closed subscheme.
		Suppose $\O_Z$ sits in the middle of a short exact sequence in $\cat Q$.
		By Lemma \ref{exact}, this means that, up to passing to an equivalent object $B \equiv \O_Z$, there is a short exact sequence in $\Coh(X)$
		\begin{align*}
			0 \to A \to B \to C \to 0.
		\end{align*}
	The goal is to show that either $C \equiv 0$ or that $B \to C$ is a weak equivalence.

		Let $\eta$ be the generic point of $Z$ and let $R = \O_{X,\eta}$ be the corresponding local ring, with maximal ideal $\m$.
		By taking germs, we get a short exact sequence 
		$$ 0 \to A_\eta \to B_\eta \to C_\eta \to 0 $$
		of $R$-modules.
		Since $B$ and $\O_Z$ are isomorphic up to $\Coh_{\leq k-1}(X)$, we have $B_\eta \simeq \O_{Z,\eta} = R/\m$.
		It follows that either $C_\eta = 0$ or $C_\eta = B_\eta$.

		If $C_\eta = 0$ then $\dim \supp C  < k$, which implies $C \equiv 0$.
		If $C_\eta = B_\eta$, then $\dim \supp A < k$, which implies $C \equiv B \equiv \O_Z$.

	To prove the other direction, let $F \in \Coh(X)$ represent a minimal element in $\cat Q$.
	Let $Z$ be the scheme-theoretic support of $F$, i.e.\ $Z = V(\Ann(F))$.
	If $\dim Z < k$, then $F \equiv 0$.
	If not, without loss of generality we may assume $Z$ is irreducible, reduced and of dimension $k$.
	This previous claim follows by minimality: we always have a surjection $F \onto F/IF$ for $I$ the ideal sheaf of a closed subscheme of $Z$.
	
	By the argument above, we already know $\O_Z$ is minimal in $\cat Q$.
	Our goal is to construct a non-zero map $\O_Z \qto F$, proving $\O_Z \equiv F$.
	For this, we use Lemma \ref{hom}.
	
	Write $i\colon Z \into X$ for the inclusion.
	We know $F = i_* F'$, where $F' = i^*F \in \Coh(Z)$.
	Let $U \subset Z$ be an affine open.
	We must have a non-torsion element $a \in F'_U$, otherwise the whole $F'$ would be torsion and $\dim Z < k$.
	View now $a$ as a non-zero map $\O_U \to F'$.

	Let $W = Z \setminus U$.
	Since $Z$ is irreducible, we must have $\dim W < k$.
	Let $\tilde U = X \setminus W$.
	We have
	\begin{align*}
		\Hom_{\tilde U}\left(\O_{\tilde U}, F_{\tilde U} \right)
		= \Hom_{\tilde U} \left( \O_{\tilde U}, i_{U*} \left( F'_U \right) \right)
		= \Hom_U \left( \O_U, F'_U \right).
	\end{align*}
	Since the map $\O_U \to F'_U$ is compatible with further restricting $\tilde U$, by Lemma \ref{hom} we have cooked up a non-zero map $\O_Z \qto F$. 
	%\todo{uhm, forse sarebbe da giustificare il fatto che l'elemento $a$ non diventa mai di torsione: cio\`e, \`e sempre che vero possiamo trovare un elemento non di torsione, ma perch\'e sempre lo stesso?}
	The claim follows. 
\end{proof}

\begin{lem}\label{small}
	The subcategory $\Coh_{\leq k}(X) \subset \Coh(X)$ is the smallest Serre subcategory containing $\Coh_{\leq k-1}(X)$ and all the sheaves $\O_Z$, with $Z \into X$ integral closed subscheme and $\dim Z = k$.
\end{lem}
\begin{proof}
	It is easier to work in the quotient $\cat Q = \Coh(X)/\Coh_{\leq k-1}(X)$.
	The claim is equivalent to showing the following.
	Suppose $\cat S \subset \cat Q$ is a Serre subcategory containing all $\O_Z$, for $Z \into X$ an integral closed subscheme with $\dim Z = k$, and suppose $F \in \Coh_{\leq k}(X)$.
	Then $F \in \cat S$.
	
	If $F = \O_Z$, with $Z$ integral, then $F \in \cat S$ by assumption.
	If $F = \O_Y$, with $Y$ reduced but possibly reducible, let $Y_1,\ldots,Y_r$ be its irreducible components, equipped with the reduced scheme structure.
	The sheaves $\O_{Y_i}$ all belong to $\cat S$ by assumption.
	There is a map $\O_Y \to \O_{Y_1} \oplus \cdots \oplus \O_{Y_r}$.
	Since the intersection of all minimal primes in a reduced ring is zero, this map is injective.
	As $\cat S$ is a Serre category, it follows that $\O_Y \in \cat S$.
	
	Suppose now $F \in \Coh_{\leq k}(X)$ satisfies the following: there exists a reduced $Y$, with ideal sheaf $I$, $\dim Y \leq k$, and $IF = 0$.
	Let $i\colon Y \into X$ be the inclusion.
	By assumption there is an $F' \in \Coh(Y)$ such that $i_*F' = F$.
	Let now $W \subset Y$ be a closed subset, with $\dim W < k$, and such that $U = Y \setminus W$ is affine.
	Let $\tilde U = X \setminus W$.
	The restriction $F'_U$ is globally generated, hence there is a surjective map $\O_U^{\oplus r} \onto F'_U$.
	By pushing forward, there is a surjection $\O_{\tilde U}^{\oplus r} \onto F_{\tilde U}$.
	Using Lemma \ref{hom}, this induces a surjection $\O_Y^{\oplus r} \onto F$ in the category $\cat Q$.
	Since $\O_Y \in \cat S$, we see that $F \in \cat S$.
	
	Suppose now $F \in \Coh_{\leq k}(X)$ is arbitrary.
	Let $Y$ be its scheme-theoretic support.
	Let $I \subset \O_X$ be the ideal sheaf defining $Y_\text{red}$.
	By noetherianity, there exists an $m$ such that $I^{m+1}F = 0$.
	We have a string of short exact sequences
	\begin{align*}
				0 \to IF  \to F \to F / I F \to 0 \\
							0 \to I^2 F \to I F \to I F / I^2 F \to 0 \\
												\cdots \\
		0 \to I^m F \to I^{m-1} F \to I^{m-1} F / I^m F \to 0 
	\end{align*}
	Now, $I(I^a F / I^{a+1} F) =0$ for any $a$, hence they belong to $\cat S$.
	Using induction starting from $I(I^m F)=0$, and using as always that $\cat S$ is Serre, we see that $F \in \cat S$.
\end{proof}

\begin{proof}[Proof of Proposition \ref{geo}]
	We proceed by induction on $k$.
	By definition, $\cat S_0$ is the smallest Serre subcategory containing all the minimal objects of $\Coh(X)$.
	By Lemma \ref{min}, these are precisely the skyscraper sheaves.
	By Lemma \ref{small}, $\cat S_0 = \Coh_{\leq 0}(X)$.
	
	Suppose the theorem is true for $k-1$, let $\cat Q = \Coh(X)/\Coh_{\leq k-1}(X) = \Coh(X)/\cat S_{k-1}$.
	Let $\cat S$ be the smallest Serre subcategory of $\cat Q$, containing all minimal objects of $\cat Q$.
	By definition, $\cat S_k$ is the kernel of $\Coh(X) \to \cat Q\to \cat Q/\cat S$.
	By Lemma \ref{min}, the minimal objects of $\cat Q$ are precisely the $\O_Z$ with $Z$ integral of dimension $k$.
	Combined with Lemma \ref{small}, we see that $\cat S$ is the image of $\Coh_{\leq k}(X)$.
	Hence, $\cat S_k = \Coh_{\leq k}(X)$.
\end{proof}
The following follows from Lemma \ref{hom}, and will be useful later.
\begin{lem}\label{residue}
Let $P$ be a non-zero minimal object in $\cat C_k$, corresponding to a (non-necessarily closed) point $x \in X$.
We have $\Hom_{\cat C_k}(P,P)=\kappa(x)$, where $\kappa(x)$ is the residue field of the point $x$.
\end{lem}
From now on we will write $\C_k = \C_k(X) = \Coh(X)/\Coh_{\leq k-1}(X)$.

%!TEX root = MASTER.tex
\section{A locally ringed space}
For this section, let $X$ be a finite-dimensional noetherian scheme, which (as per our blanket conventions) is moreover defined over our base ring $\k$.
Our present goal is to define an auxiliary locally ringed space $S = S(X,k,L)$, depending on $X$, an integer $k$ and a sheaf $L$.
This space will control the isomorphism type of $X$ up to subsets of dimension $k-1$. %\todo{o $\dim X -k -1$?}
A posteriori it will be obvious that $S$ does not depend on the sheaf $L$.
The reason we initially insist on the dependency on $L$ will become clear in the next section: an equivalence $\C_k(X) \simeq \C_k(Y)$ need not send $\O_X$ to $\O_Y$.

We say a point $x \in X$ has \emph{dimension $k$} if $\dim \overline{\{x\}} = k$.
Put differently, $x$ is the generic point of a subvariety of dimension $k$, which (when $X$ is equidimensional and catenary) in turn is equivalent to $\dim \O_{X,x} = \dim X - k$.
We write $X_{\geq k} \subset X$ for the subset of points of dimension at least $k$.
In particular, $X_{\geq k}$ always contains the generic points of the irreducible components (of dimension at least $k$) of $X$.
We endow $X_{\geq k}$ with the subspace topology.
If $i\colon X_{\geq k} \to X$ denotes the inclusion, we may view the former as a locally ringed space with structure sheaf $i^{-1}\O_X$.
This space will turn out to be isomorphic to the space $S$ we are about to define.

Before we proceed, recall that $i^{-1}\O_X$ is (by definition) obtained by sheafifying the presheaf $i^+\O_X$, given by
\begin{align*}
	V \mapsto \varinjlim_{U \supset V} \O_X(U)
\end{align*}
where the limit ranges over all $U \subseteq X$ open and containing $V \subseteq X_{\geq k}$.
\begin{lem}\label{nosheaf}
	The presheaf $i^+\O_X$ is already a sheaf, hence
	\begin{align*}
		i^{-1}\O_X(V) = \varinjlim_{U \supset V} \O_X(U).
	\end{align*}
\end{lem}
\begin{proof}
	Let $V \subseteq X_{\geq k}$ be open.
	A section in $i^+\O_X(V)$ is represented by an $\alpha \in \O_X(U)$, for $U \subseteq X$ open and $V \subset U$.
	Let $\{V_i\}_i$ be an open cover of $V$, such that the restriction $\alpha|V_i = 0$ vanishes, for all $i$.
	We may then find open subsets $U_i \subset U$, such that $U_i \cap X_{\geq k} = V_i$ and $\alpha|U_i = 0$.
	Let $U'$ be the union of the $U_i$.
	Then, $V \subset U' \subset U$ and $\alpha|U' = 0$.
	Hence the section $[\alpha] \in i^+\O_X(V)$ must be zero.
	
	Suppose once again $\{V_i\}_i$ is an open cover of a given open $V \subset X_{\geq k}$.
	We may assume the cover to be finite.
	
	Let $[\alpha_i] \in i^+\O_X(V_i)$ be a collection of sections such that $[\alpha_i]|V_i \cap V_j = [\alpha_j]|V_i \cap V_j$.
	For each $i$, suppose $\alpha_i \in \O_X(U_i)$ for $V_i \subset U_i \subset X$.
	Now, for each pair $ij$, we may choose an open subset $V_i \cap V_j \subset U_{ij} \subset U_i \cap U_j$, such that $\alpha_i|U_{ij} = \alpha_j|U_{ij}$. Note that $W_{ij} = U_i \cap U_j \setminus U_{ij}$ has dimension at most $k-1$. 
	We replace $X$ with $U = \cup_{i} U_{i} \setminus \cup_{i,j} \overline{W_{ij}} $, which is open and contains $X_{\geq k}$. Finally, we replace the open subsets $U_i$ with $U'_i = U_i \cap U$.
	
	Since $\O_X$ is a sheaf, and $U'_i \cap U'_j \subset U_{ij}$, there exists $\alpha \in U$ restricting to $\alpha_i$ on $U_i'$, which gives the desired section as the complement of $U$ in $X$ has dimension at most $k-1$.
	
%	Consider now the set $Z$ defined as the union of all $\overline{\{x\}}$, where $x \in X$ is such that $\dim \overline{\{x\}} \leq k$, and there exist $i,j$ for which $x \notin U_{ij}$.
%	This set is closed by noetherianity. %\todo{va giustificato?}
%	For each $i$, define $W_i$ to be the union of all $\overline{\{x\}}$, where $x \in X_{\geq k} \setminus U_i$.
%	The $W_i$ are also closed.
%	Finally, define $U'_i = X \setminus (Z \cup W_i)$.
%	It follows that $U'_i \subset U_i$ is open, $U'_i \cap U'_j \subset U_{ij}$, and $V \subset \bigcup_{i} U_i'$.
%	Since $\O_X$ is a sheaf, there exists $\alpha \in U'$ restricting to $\alpha_i$ on $U_i'$.
%	In particular, $[\alpha] \in i^+\O_X(V)$ restricts to $[\alpha_i]$.
%	Hence, $i^+\O_X$ is a sheaf, and $i^+\O_X = i^{-1}\O_X$.
\end{proof}

\subsection{The topological space $S$}
We now come to the definition of $S$.
Let $k \geq 0$ be an integer.
As a set, $S$ consists of isomorphism classes of (non-zero) minimal objects of $\C_d(X)$, where $d$ ranges between $k$ and $\dim(X)$:
\begin{align*}
	S \coloneq \bigcup_{d \geq k} \left\lbrace 0 \neq P \in \cat C_d(X) \mid P \mbox{ minimal} \right\rbrace_{/_\text{iso}}
\end{align*}
We define a topology on $S$, by declaring what the closure of a point $s \in S$ ought to be. 
Let $P \in \Coh(X)$ represent a nonzero minimal object in $\C_d(X)$, with $d \geq k$.
Define
\begin{align*}
	Z_P \coloneq 
		\bigcup_{d \geq j \geq k} 
		\bigcap_{
			\substack{ P' \in \C_j(X) \\ P' \text{ represents } P}
			}
		\left\{ 0 \neq Q \in \C_j(X) \mid Q \text{ is minimal, } Q \text{ is a quotient of } P' \right\}_{/_\text{iso}}
\end{align*}
We endow $S$ with the coarsest topology containing the sets $S \setminus Z_P$, for all $P \in \Coh(X)$.
\begin{lem}\label{homeo}
	Define a map $X_{\geq k} \to S$ by sending $x \mapsto \O_{\overline{\{x\}}}$, where $\overline{\{x\}}$ is equipped with the reduced scheme structure.
	This map is a homeomorphism.
\end{lem}
\begin{proof}
	Since $X$ is noetherian, the subsets $X \setminus \overline{\{x\}}$, as $x$ varies in $X$, form a basis of the topology.
	Proposition \ref{geo} and Lemma \ref{min} allow us to conclude.
\end{proof}

\subsection{A Picard group}\label{picardo}
We introduce now the analogue of $\Pic(X)$ for the category $\C_k(X)$.
We define $\Pic_{\geq k}(X)$ to be the group (under tensor product) of isomorphism classes of objects $L_1 \in \C_k(X)$ for which there exists a representative $L \in \Coh(X)$ satisfying the following property: there exists an open subset $U \subset X$, such that $X_{\geq k}\subset U$, and the restriction $L_U$ is an invertible sheaf.
Put concisely: $\Pic_{\geq k}(X)$ consists of (the $\C_k(X)$-isomorphism classes of) those $L$ which are line bundles away from a closed subset of dimension at most $k-1$. We will see in proposition \ref{equivlrs} that this definition is intrinsic to the category $\C_k(X)$.

\subsection{The locally ringed space $S$}
For $F \in \Coh(X)$, we define its \emph{topological $k$-support} to be $\supp_k F \coloneq X_{\geq k} \cap \supp F$.
Using the homeomorphism of Lemma \ref{homeo}, we may view $\supp_k F$ as a closed subset of $S$.
Similarly, if $G \in \C_d(X)$ with $d \leq k$, we define $\supp_k G = \supp_k \tilde G$, where $\tilde G \in \Coh(X)$ is any representative of $G$. 

Fix now $L \in \Pic_{\geq k}(X)$.
We will now define a sheaf of rings $\O_S$ on $S$, depending on $L$.
For $V \subset S$ an arbitrary open subset, we define
\begin{align*}
	\O_S(V) \coloneq \varinjlim_{U \supset V} \Hom_{U} (L_U,L_U)
\end{align*}
where $U \supset V$ runs over all open subsets of $X$ containing $V$.
\begin{prop}
	The assignment $V \mapsto \O_S(V)$ defines a sheaf of rings on $S$, making $(S,\O_S)$ into a locally ringed space.
	Moreover, the homeomorphism $X_{\geq k} \to S$ from Lemma \ref{homeo} may be upgraded to an isomorphism of locally ringed spaces over $\Spec \k$.
\end{prop}
\begin{proof}
	Suppose first $L = \O_X$.
	Then Lemma \ref{homeo} and Lemma \ref{nosheaf} imply the claim.
	Let $L$ be general.
	There exists then an open subset $U$, with $X_{\geq k} \subset U$, such that $L$ becomes invertible once restricted to $U$.
	But then the natural map $\O_X \to \lHom_X(L,L)$ also becomes an isomorphism once restricted to $U$.
	The claim then follows from the case where $L = \O_X$.
\end{proof}

\subsection{Intermezzo}
To proceed with the final proof of this section, we must first introduce one last category.
Suppose $0 \neq P \in \C_k(X)$ is minimal, and denote by $x \in X_{\geq k}$ the corresponding point.
In particular, if $Z = \overline{\{x\}}$, the sheaf $\O_Z$ represents $P$.
The goal is to define a category $\C_P$, recovering the local ring $\O_{X,x}$ at $x$.

\begin{lem}
Given an element $F \in \Coh(X)$, the subset $\supp_k(F) \subset X_{\geq k}$ only depends on the class of $F$ in $\cat C_k$.
\end{lem}
\begin{proof}
Let $P$ be a minimal object in $\cat C_k$. Then by lemma (\ref{hom}) the corresponding point $P \in X_{\geq k}$ belongs to $\supp_k(F)$ if and only if $\Hom_{\cat C_k}(F,P) \neq 0$.
\end{proof}

Consider the collection of $E \in S$ such that $P \nin \supp_k E$, and let $\cat S_P \subset \C_k(X)$ be the smallest Serre subcategory containing them.
We define $\C_P$ to be the quotient $\C_k(X)/\cat S_P$. %\todo{Bob, controlla per favore. Ho cercato di semplificare, ma ci sta che abbia barato.}
Furthermore, write $E_P$ for the class of $E$ in $\C_P$.
\begin{lem}\label{cplocale}
	Let $L \in \Pic_{\geq k}(X)$.
	We may identify the ring $\End_{\C_P}(L)$ with the local ring $\O_{X,x}$.
	Moreover, the functor $\C_P \to \mod(\O_{X,x})$ sending $E$ to $\Hom_{\C_P}(L,E)$ is an equivalence.
\end{lem}
In particular we may identify the stalk $E_x \in \mod(\O_{X,x})$ with the object $E_P \in \C_P$.
\begin{proof}
	Let $\Sigma_x$ be the collection of closed subsets of $X$ not containing $x$.
	From the definition of $\cat S_P$, we see that $\Coh(X)/\cat S_{\Sigma_x} = \C_P$. %\todo{sta cosa non l'ho verificata ma ci credo. Magari va giustificata nella dimostrazione?}
	Using Lemma \ref{hom}, we see
	\begin{align*}
		\Hom_{\C_P}(E,F) = \varinjlim_{x \in U \subset X} \Hom_U(E_U,F_U) = \Hom_{\O_{X,x}}(E_x,F_x).
	\end{align*}
	Since $L_x = \O_{X,x}$, and $\O_{X,x} = \End_{\O_{X,x}}(\O_{X,x}) = \End_{\C_P}(L)$.
	Hence the functor is fully faithful.
	To prove it is essentially surjective, it suffices to observe that, given a finitely generated $\O_{X,x}$-module $M$, there exists a coherent sheaf $E \in \Coh(X)$ such that $E_x = M$.
\end{proof}
\begin{rmk}\label{funziona}
	When $X$ in the above lemma is integral, and $x$ is its generic point, then $\O_{X,x} = K(X)$ is its function field.
	It then follows that $\C_{\dim X}(X) = \C_{\O_X}$, and the functor $\C_{\O_X} \to \mod(K(X))$ is an equivalence.
	Hence $\C_{\dim X}(X)$ captures the birational type of $X$.
\end{rmk}

\subsection{A non-commutative remark}
Following up Remark \ref{funziona}, we briefly discuss a non-commutative avenue.
One perspective on non-commutative geometry is that a non-commutative space should be given by an abelian category $\cat A$, satisfying some niceness properties.
Much work has been devoted to this point of view, see for example \cite{artinzhang,artintatevdb,rosenbergncag,vdbblow}.
A basic question is then whether there exists a non-commutative analogue of birational geometry.
For example in \cite{smith} one finds a candidate for a function field of a general $\cat A$, and in \cite{presotto} there are interesting examples of birational non-commutative surfaces.

However, a complete and satisfactory theory of non-commutative birational geometry does not seem to exist presently.
For instance, to capture higher dimensional phenomena (such as flops and flips), one needs to know not just when two spaces are merely birational, but also when they are isomorphic in a certain codimension.

On the other hand, as we discussed in the previous section, the categories $\Coh_{\leq k}(X)$ are \emph{intrinsic} to the category $\Coh(X)$.
Indeed, given an abelian category $\cat A$, we may form our sequence of quotients $\cat A_k$ as in \eqref{quotients_of_cats}.
In light of Remark \ref{funziona}, if there exists an $n$ such that $\cat A_n \neq 0$ but $\cat A_{n+1} = 0$, we would view $\cat A_n$ as the \emph{non-commutative function field} of $\cat A$.
It would be interesting to see the relation between this and Smith's function field \cite{smith}.

In general, given the main theorem of this paper, the quotient categories $\cat A_k$ (for $k \geq n$) could be seen as capturing the non-commutative space $\cat A$, up to a certain codimension.

\subsection{Equivalences}
We conclude this section by showing that equivalences of the quotient categories induce isomorphism of the locally ringed spaces.
\begin{prop}\label{equivlrs}
	Let $X$, $Y$ be schemes of finite type over $\k$.
	A $\k$-linear equivalence $\Phi\colon \C_k(X) \to \C_k(Y)$ induces an isomorphism of locally ringed spaces $\phi\colon X_{\geq k} \to Y_{\geq k}$.
\end{prop}
\begin{proof}
	Let $L \in \Coh(Y)$ be a representative for $\Phi(\O_X)$.
	We want to show that $L$ is a line bundle away from a small enough closed subset.
	Notice first that $\O_X$ satisfies the following property: if $P$ is any non-zero minimal object of $\C_k(X)$, then $\Hom_{\C_k(X)}(\O_X,P)$ is a one-dimensional $\End_{\C_k(X)}(P)$ vector space.
	Moreover, $\O_X$ is \emph{locally maximal} with respect to this property: given $F$ such that $\Hom_{\C_k(X)}(F,P)$ is one-dimensional, any surjection $F \to \O_X$ in $\C_P$ must be an isomorphism. %\todo{va giustificato?}
	Since these two properties are categorical (i.e. intrinsic to $\C_k(X)$ and the minimal object $P$) they will satisfied by $L$ as well.
	
	We claim that, for any sheaf $L$ satisfying the two properties above, there exists an open subset $U \subset Y$, with $Y_{\geq k} \subset U$, such that the restriction $L_U$ is a line bundle.
	Indeed,	consider the function $\beta\colon Y \to \N$, taking $y \in Y$ to the rank of the fibre $\dim_{\kappa(y)} L \otimes \kappa(y)$.
	This function is upper semicontinuous.
	The locus $\{\beta = 0 \}$ is therefore open.
	Let $Z$ be the union of the irreducible components of dimension strictly less than $k$.
	We see that $\{\beta = 0\} \cap (Y \setminus Z)$ is empty.
	Therefore, the locus $U = \{ \beta = 1 \} \cap (Y \setminus Z)$ is open, and (by assumption) contains all points of dimension $k$.
	By Nakayama, the $\O_{Y,y}$-module $L_{y}$ has rank $1$ at all points $y$ in $U$.
	Using Lemma \ref{cplocale}, and the local maximality property of $L$, we deduce that $L_y=\O_{Y,y}$ at all points of dimension $k$, as there is by definition a surjection $\O_{Y,y}\rightarrow L_y$, which must be an isomorphism by maximality.
	Finally, the set of points $y$ such that $L_y$ is free is open.
	Since it contains all points of dimension $k$ it follows that $L_y$ is free of rank $1$ at all points in a subset $U'$ containing all points of dimension $k$.
	
	From the discussion above it follows that the equivalence $\Phi$ induces an isomorphism of locally ringed spaces $S(X,k,\O_X) \simeq S(Y,k,\Phi(\O_X))$.
	The proposition then follows: the former is isomorphic to $X_{\geq k}$, and the latter to $Y_{\geq k}$, as locally ringed spaces.
\end{proof}

%!TEX root = MASTER.tex
\section{Gabriel's theorem}\label{gabrielsection}
As usual, let $\k$ be our base ring, and let $X$, $Y$ be schemes over it.
Recall that we write $\cat C_k(X)$ for the quotient $\Coh(X)/\Coh_{\leq k-1}(X)$, with analogous notation for $Y$.
To prove our main theorem we will deal with the two directions separately.
\subsection{The ``hard'' direction}
Suppose we have an equivalence of categories $\cat C_k(X) \simeq \cat C_k(Y)$, which by default is assumed linear over $\k$.
Proposition \ref{equivlrs} implies there is an isomorphism of locally ringed spaces $X_{\geq k} \simeq Y_{\geq k}$ over $\Spec \k$.
\begin{prop}\label{puppa}
	Let $\phi\colon X_{\geq k} \to Y_{\geq k}$ be an isomorphism of locally ringed spaces over $\Spec \k$.
	There exist open subsets $U \subset X, V \subset Y$ containing all points of dimension $\geq k$, and an isomorphism $f \colon U \rightarrow V$ of $\k$-schemes, which restricts to $\phi$.
\end{prop}
Let us start with a lemma.
\begin{lem}\label{lociso}
	Let $X, Y$ be schemes of finite type over $\k$.
	Assume there is an isomorphism of $\k$-algebras between local rings $\O_{X,p} \simeq \O_{Y,q}$ where $p \in X$, $q \in Y$ are not necessarily closed.
	Then there are open subschemes $p \in U \subset X$, $q \in V \subset Y$ such that $U \simeq V$ as $\k$-schemes.
\end{lem}
\begin{proof}
Fix open neighborhoods $p \in \Spec(A) \subset X, q \in \Spec(B)\subset Y$, and consider the composition $B \rightarrow B_{q} \rightarrow A_p$.
Given a finite set of generators $y_1, \ldots, y_n$ for $B$ as a $\k$-algebra, their images will be in the form $x_1/s_1, \ldots , x_n/s_n$ where $x_1,\ldots,x_n$ belong to $A$ and $s_1,\ldots,s_n$ belong to $A \setminus I(p)$.
Here $I(p) \subset A$ denotes the prime ideal corresponding to the point $p \in \Spec A$.

The image of $B$ is thus contained in $S^{-1}A$ where $S=\lbrace s_1^{i_1} \ldots s_n^{i_n}\rbrace$.
The inclusion $\Spec(S^{-1}A)\subset \Spec(A)$ is open and contains the point $p$.
Hence we just constructed a map $W \rightarrow \Spec(B)$ where $W$ is a neighborhood of $p$, sending $p$ to $q$.
By restricting to open subsets $p \in U \subset \Spec(A)$ and $q \in V \subset \Spec(B)$, we may further assume this map to be unramified, quasi-finite, and of degree one at all points, i.e. an isomorphism.
\end{proof}

\begin{lem}\label{uniqiso}
Let $U=\Spec(A), V=\Spec(B)$ be affine schemes of finite type over $\k$, and let $f, g \colon U \rightarrow V$ be two morphisms.
Suppose further that there is a $p \in U$ such that $f(p)=g(p)=q$, and $f^*_{\mid\O_{V,q}}=g^*_{\mid \O_{V,q}}$.
Then there is an open subscheme $p \in U' \subset U$ such that $f$ and $g$ are equal when restricted to $U'$. 
\end{lem}
\begin{proof}
Let $y_1, \ldots, y_n$ be a set of generators for $B$ as a $\k$-algebra.
We know $f(x_i)=g(x_i)$ in the localization $A_{I(p)}$ for all $i$.
After inverting a finite number of elements in $A \setminus I(p)$ we will have $f(x_i)=g(x_i)$ in $S^{-1}A$, for some (finite) set $S \subset (A \setminus I(p))$.
To conclude, set $U' = \Spec(S^{-1}A) \subset \Spec(A)$.
\end{proof}

\begin{proof}[Proof of Proposition \ref{puppa}]
Let $V \subset Y$ be an affine open subset.
As $Y$ is of finite type over $\k$, the algebra $\O_Y(V)$ is generated by a finite set $y_1, \ldots , y_n$ of generators.
Consider $y_1, \ldots, y_n$ as elements of $\O_{Y_{\geq k}}(V \cap Y_{\geq k})$.
There exists an open subset $U \subset X$, containing the inverse image of $V \cap Y_{\geq k}$, such that the elements $\phi^*(y_1), \ldots, \phi^*(y_n)$ are represented by $x_1, \ldots, x_n \in \O_X(U)$.
This induces a morphism $\phi_{U}\colon U \rightarrow V$. 

The morphism $\phi_{U}$ is an isomorphism at all the local rings of points $p \in V \cap Y_{\geq k}$.
By Lemmas \ref{lociso}, \ref{uniqiso} given any point $p \in U \cap Y_{\geq k}$ there is a neighborhood $U_p \subset U$ such that $(\phi_{U})_{\mid U_{p}}$ is an isomorphism.
Up to restricting $U$ and $V$ to smaller open subsets containing all points of $X_{\geq k} \cap U$ and $Y_{\geq k} \cap V$, we may then assume that $\phi_{U}$ is an isomorphism.

Note that, following the construction above, the morphism $\phi_U$ coincides with the original $\phi$ on all local rings in $U \cap X_{\geq k}$.
Explicitly, for each $p \in X_{\geq k}$ we have $\O_{X_{\geq k}, p} = \O_{X,p}$, and $\phi_{U,p} = \phi_p$.

Suppose now we apply the same construction to two different open subsets, obtaining isomorphisms $\phi_{U_i}\colon U_i \rightarrow V_i,\phi_{U_j}\colon U_j \rightarrow V_j$.
The two maps must agree on local rings for all $p \in U_i \cap U_j \cap Y_{\geq k}$. 
Lemma \ref{uniqiso} informs us that, up to removing a closed subset of dimension at most $k-1$ from $U_i$ and $ U_j$ and their images, the two maps must agree.
Pick now a finite open affine cover $U_1,\ldots,U_n$ of $X$.
We may iterate the procedure above to all $m$-fold intersections, with $m \leq n$, refining our open cover and obtaining the claim.
\end{proof}

\subsection{The ``easy'' direction}
The converse direction, i.e. that a birational map induces an equivalence of categories, is actually false in the generality we have maintained so far.
\begin{rmk}
	Let $X$ be the spectrum of a DVR, with closed point $x$ and open point $\eta$.
	Let $K = \O_{X,\eta}$ be the local ring at $\eta$.
	Let $Y$ be $\Spec K$, and write $y$ for the unique point of $Y$.
	Then $X$ and $Y$ are birational, in the sense that $\O_{Y,y} \simeq \O_{X,\eta}$, but $X_{\geq 1} \nsimeq Y_{\geq 1}$.
	Indeed, the former consists of just $\eta$, while the latter is empty.
	At the categorical level, $\cat C_1(X) \simeq \mod(K)$ while $\cat C_1(Y) = 0$.
	The culprit is the fact that the singleton $\{ \eta \}$ is \emph{open} in $X$, but has smaller dimension.
	This phenomenon does not occur for varieties.
\end{rmk}
However, over a field everything works just fine.
\begin{prop}
	Suppose our base ring $\k$ is a finite type algebra over a field.
	Let $j\colon U \into X$ be an open immersion.
	Suppose $\dim X \setminus U < k$.
	Then $j^*$ induces an equivalence $\C_k(X) \simeq \C_k(U)$.
\end{prop}
\begin{proof}
	This is a consequence of \cite[\href{http://stacks.math.columbia.edu/tag/01PI}{Tag 01PI}]{sp} and Lemma \ref{hom}.
\end{proof}%\todo{ora va bene?}

\begin{rmk}\label{Rydh}
The hypotheses in the previous proposition con be weakened. The statement holds true as long as for an open subset $U$ we always have that $U_{\geq k} = X_{\geq k} \cap U$, which by \cite[10.6.2]{EGAIV} is implied by $X$ being Jacobson, universally catenary and every irreducible component of $X$ being equicodimensional.
\end{rmk}

\subsection{The main theorem}
Assembling our previous results together, we arrive at our main result.
We say two schemes $X,Y$ are \emph{isomorphic outside of dimension $k-1$} of there exist isomorphic open subsets $U \subset X, V \subset Y$, such that $X_{\geq k} \subset U$, and $Y_{\geq k} \subset V$. %\todo{giusto?}
\begin{thrm}\label{suca}
	Assume our base ring $\k$ is a finite type algebra over a field.
	Let $X,Y$ be schemes of finite type over $\k$.
	Then $X$ and $Y$ are isomorphic (over $\k$) outside of dimension $k-1$ if and only if $\cat C_k(X) \simeq \cat C_k(Y)$.
\end{thrm}
Assume $\k$ is a field, and $X$ and $Y$ are irreducible.
In this context it is preferable to rephrase things in terms of codimension.
Let $\cat C^c(X) = \cat C_{\dim X - c}(X)$, and similarly for $Y$.
Recall that $X$ and $Y$ are \emph{isomorphic in codimension $c$} if there exist isomorphic open subsets $U \subset X$, $V \subset Y$, such that $\codim X \setminus U > c$, $\codim Y \setminus V > c$.
\begin{cor}
	Let $\k$ be a field, and let $X$ and $Y$ be irreducible (but not necessarily reduced).
	Then $X$ and $Y$ are isomorphic in codimension $c$ if and only if $\cat C^c(X)$ and $\cat C^c(Y)$ are equivalent.
\end{cor}

%!TEX root = MASTER.tex
\section{The group of autoequivalences}
Gabriel originally also characterized the autoequivalences of the category of coherent sheaves.
Indeed, $\Aut(\Coh(X)) = \Aut(X) \ltimes \Pic(X)$, where to form the semi-direct product we use the standard (left) action of automorphisms on line bundles: $(f,L) \mapsto f_* L = (f^{-1})^*L$.
We want to now describe what happens when passing to the quotients.

For this section, assume $\k$ is a finite type algebra over a field.
Let $\Aut_{\geq k}(X)$ be the group of birational self-maps $X \dashrightarrow X$ which are defined on an open subset containing $X_{\geq k}$, modulo the obvious equivalence relation: $f \sim g$ if the two are equal on an open subset containing all points of dimension at least $k$.
Recall the definition of $\Pic_{\geq k}(X)$ from \ref{picardo}.

\begin{thrm}\label{autoteorema}
There is an isomorphism of groups 
$$ \Aut(\cat C_k(X)) \simeq \Aut_{\geq k}(X)\ltimes \Pic_{\geq k}(X). $$
Where to form the semi-direct product the action is given by $(f,L) \mapsto (f^{-1})^*L$.
\end{thrm}
Let $G = \Aut_{\geq k}(X)\ltimes \Pic_{\geq k}(X)$.
There is an obvious group homomorphism $G \to \Aut(\C_k)$, sending the pair $(f,L)$ to the auto-equivalence $F \mapsto (f^{-1})^* F \otimes L$.
Conversely, an auto-equivalence $\Phi$ induces an automorphism of the locally ringed space $X_{\geq k}$, and hence a birational self-map $f$, which is uniquely defined up to a subset of dimension at most $k-1$ thanks to lemma \ref{uniqiso}. 
One checks that the assignment $\Phi \mapsto (f,\Phi(\O_X))$ is also a group homomorphism.
Clearly, the composition $G \to \Aut(\C_k(X)) \to G$ is the identity.
We thus have a (split) short exact sequence
\begin{align*}
	1 \to Q \to \Aut(\C_k(X)) \to \Aut_{\geq k}(X)\ltimes \Pic_{\geq k}(X) \to 1.
\end{align*}
Proving the theorem amounts to showing that the kernel $Q$ is trivial.
To this end, suppose we have an auto-equivalence $\Phi$ such that the induced map on $X_{\geq k}$ is the identity, and $\Phi(\O_X) = \O_X$.
We will construct a natural isomorphism $\tau\colon \Phi \to \Id_{\C_k(X)}$.

\begin{lem}\label{eximodul}
let $M,M'$ be coherent sheaves on $X$ and assume that $M\otimes \O_{X,P}=M_P = M'_P$ for a point $P \in X$. Then there exists an open neighborhood $U$ of $P$ and an isomorphism of coherent sheaves $M_U \simeq M'_U$.
\end{lem}
\begin{proof}
The same argument we used for Lemma \ref{lociso} works here.
\end{proof}

\begin{lem}\label{uniqmodul}
let $M,M'$ be coherent sheaves on $X$ and let $f,g\colon M \rightarrow M'$ be two morphisms.
If the localizations $f_x, g_x\colon M_x \rightarrow M'_x$ are equal for all points $x \in X_{\geq k}$, then there is an open subset $X_{\geq k} \subset U$ such that the restrictions $f|_U, g|_U$ are equal.
\end{lem}
\begin{proof}
The same argument as in Lemma \ref{uniqiso} shows that given each such $x$ there is an open neighborhood containing $x$ where the maps are equal.
We can then conclude by the fact that being $0$ is an open property.
\end{proof}

\begin{proof}[Proof of Theorem \ref{autoteorema}]
We wish to construct a natural isomorphism $\tau\colon \Phi \to \Id_{\C_k(X)}$, as explained earlier.
To start, assume that $X$ is affine.
Let $M$ be a coherent sheaf on $X$.
Then $$\Hom_{\cat C_k(X)}(\O_X, M)=\Hom_{\cat C_k(X)}(\Phi(\O_X),\Phi(M))=\Hom_{\cat C_k(X)}(\O_X,\Phi(M)).$$

Fix a representative $M' \in \Coh(X)$ of $\Phi(M)$, and let $m'_1,\ldots,m'_r$ generators for $M'$. To each of these corresponds a map $s'_i\colon \O_X \rightarrow M'$. Let $s_1, \ldots, s_r$ be the inverse images of these maps in $\Hom_{\cat C_k(X)}(\O_X, M)$. By Lemma \ref{hom} we can pick morphisms $S_1, \ldots, S_r \in \Hom_U(\O_U,M_U)$ where $U$ is an open subset of $X$ whose complement has dimension at most $k-1$, which in turn give us global sections $m_1, \ldots, m_r$ of $M_U$. 

Note that the assignment $m'_i \mapsto m_i$ induces an isomorphism between the local modules $M'_P$ and $M_P$ for all points $P$ of dimension $k$.
By Lemma \ref{eximodul}, up to restricting $U$ by removing a subset of dimension at most $k-1$, we can assume that the sheaf of relations between $m'_1, \ldots, m'_r$ is isomorphic to the sheaf of relations between $m_1, \ldots, m_r$, and thus the assignment $m'_i \mapsto m_i$ induces a morphism $M'_U \rightarrow M_U$ (note that this works because every open subset is sent to itself by the induced morphism on $X_{\geq k}$).

Now fix an affine covering $U_1, \ldots, U_s$ of $U$.
The assignment $m'_1, \ldots, m'_r \mapsto m_1, \ldots, m_r$ induces a morphism $M'_{U_i} \rightarrow M_{U_i}$.
Thanks to Lemma \ref{uniqmodul}, up to restricting $U$ further, we can assume that these maps glue to a morphism $M'_U \rightarrow M_U$.
Note now that $m_1, \ldots, m_r$ are generators of $M$ on an appropriate open subset due to the exact sequence $$\O_X^r \xrightarrow{\oplus s'_i} M' \rightarrow 0$$ being sent to $$\O_X^r \xrightarrow{\oplus s_i} M \rightarrow 0$$ by $\psi^{-1}$. So doing the same process we can construct a map $M_V \rightarrow M'_V$ for some open subset $V$ of $X$ whose complement has dimension at most $k-1$.
On $U \cap V$ these two maps are inverse to each other, so we have constructed an isomorphism $\tau_M\colon M \simeq M'$ in $\cat C_k(X)$.

Note now that on local rings this isomorphism is uniquely determined by the isomorphism $\Hom_{\cat C_P}(\O_X,M) \simeq \Hom_{\cat C_P}(\O_X,M')$, for each minimal object $P$.
Using Lemmas \ref{uniqmodul}, \ref{hom} we conclude that the isomorphism $\tau_M$ is unique and functorial.
We have thus proven the claim for affine schemes.
Now note that by taking an affine covering $X_i$ of a general scheme $X$ of finite type over $\k$, the maps $\tau_{M,i}$ we construct as above will glue (up to removing closed subsets of dimension at most $k-1$), giving rise to the sought after natural isomorphism $\tau$.
\end{proof}

\bibliographystyle{fabio}
\bibliography{bib}

\end{document}